\documentclass[article, 11pt]{amsart}
\usepackage{amsmath, amssymb, fontenc, amsthm, array, tikz, hyperref, amsfonts, enumitem, mathtools, stmaryrd, centernot}
\usepackage[english]{babel}
\usepackage[toc]{appendix}
\usepackage{bigfoot}
\usepackage{doi}
\usepackage[margin=1in]{geometry}
\usetikzlibrary{arrows}

\tikzset{degil/.style={
            decoration={markings,
            mark= at position 0.5 with {
                  \node[transform shape] (tempnode) {$\subseteq$};
                  }
              },
              postaction={decorate}
}
}

\makeatletter
\newcommand{\preceqdot}{\mathrel{\mathpalette\pr@ceqd@t\relax}}
\newcommand{\pr@ceqd@t}[2]{%
  \begingroup
  \sbox\z@{$#1\prec$}\sbox\tw@{$#1\preceq$}%
  \dimen@=\dimexpr\ht\tw@-\ht\z@\relax
  {\preceq}%
  \mkern-5mu
  \raisebox{\dimen@}{$\m@th#1\cdot$}%
  \endgroup
}
\makeatother

\providecommand{\keywords}[1]{  \small  \textbf{\textit{Keywords:}} #1 \normalsize}

\hypersetup{
    colorlinks = true,
    citecolor = blue,
    urlcolor = blue,
    linkcolor = red
}

\theoremstyle{plain}
\newtheorem{theorem}{Theorem}[section]
\newtheorem{lemma}[theorem]{Lemma}
\newtheorem{corollary}[theorem]{Corollary}
\newtheorem{proposition}[theorem]{Proposition}
\newtheorem{remark}[theorem]{Remark}

\theoremstyle{definition}
\newtheorem{definition}[theorem]{Definition}


\begin{document}

\title{Embeddings, ultrapowers and direct powers}
\author{Pedro Teixeira Yago}
\address{Orcid: \href{}{0000-0001-7993-4516}, Classe di Lettere e Filosofia, Scuola Normale Superiore di Pisa, Italy}
\email{pedro.tyago@outlook.com}
\date{}

\begin{abstract}
We study when embeddings lift from structures to their ultrapowers, and when an ultrapower embeds into its direct power. We offer a refinement of Blass's theorem on the Rudin-Keisler order and introduce cardinal invariants $\tau_\mathcal{A}$ which imply the embeddability $\mathcal{A}^I / \mathcal{U} \hookrightarrow \mathcal{A}^I$.
\end{abstract}

\maketitle


\section{Introduction}\label{introduction}

A fundamental question about ultrapowers concerns the extent to which structural relationships between given structures, such as embeddings, are lifted to their ultrapowers. This paper addresses two interconnected problems: first, when an embedding between structures lifts to an embedding between their ultrapowers; and second, when an ultrapower admits an embedding into its corresponding direct power.

The first problem is intimately connected with the Rudin-Keisler ordering of ultrafilters \cite{rudin1971} \cite{shelah} \cite{blass1973} \cite{keisler1967}. Recall that for ultrafilters $\mathcal{U}_I$ over $I$ and $\mathcal{U}_J$ over $J$, we write $\mathcal{U}_I \leq_\mathrm{RK} \mathcal{U}_J$ if there exists a map $h : J \to I$ such that $X \in \mathcal{U}_I \Leftrightarrow h^{-1}(X) \in \mathcal{U}_J$. A well-known theorem of Andreas Blass \cite{blass} states that $\mathcal{U}_I \leq_\mathrm{RK} \mathcal{U}_J$ iff for every structure $\mathcal{A}$, the $\mathcal{A}^I / \mathcal{U}_I \preceq \mathcal{A}^J / \mathcal{U}_J$. This immediately gives a sufficient condition for the lifting of embeddings: if $\mathcal{A} \hookrightarrow \mathcal{B}$ and $\mathcal{U}_I \leq_\mathrm{RK} \mathcal{U}_J$, then $\mathcal{A}^I / \mathcal{U}_I \leq_\mathrm{RK} \mathcal{B}^J / \mathcal{U}_J$.

However, the Rudin-Keisler criterion is not necessary in general. The question of whether there is a generalization to the criterion is thus raised. A partial answer is offered by a refinement of Andreas Blass's result \cite{blass}:

\setcounter{section}{2}
\setcounter{theorem}{9}

\begin{theorem}
Let $u: \mathcal{A} \hookrightarrow \mathcal{B}$, $h: J \to I$ be a mapping, and $\mathcal{U}_I$ and $\mathcal{U}_J$ be ultrafilters respectively over $I$ and $J$. Let also $H: X \mapsto h^{-1}[X]$ and define $w: A^I \to B^J$ such that, for each $f \in A^I$, $w(f) \in B^J$ is the function such that $\{j \in J \mid w(f)(j) = u(a)\} = H(\{j \in I \mid f(j) = a\})$.\footnote{As we later note, $w(f)$ is indeed well-defined, that is, its description is unique.} Define $e: \mathcal{A}^I / \mathcal{U}_I \to \mathcal{B}^J / \mathcal{U}_J ; [f]_{\mathcal{U}_I} \mapsto [w(f)]_{\mathcal{U}_J}$. Then $e$ is an embedding iff $X \in \mathcal{U}_I \Leftrightarrow h^{-1}(X) \in \mathcal{U}_J$.
\end{theorem}

The second problem -- when $\mathcal{A}^I / \mathcal{U} \hookrightarrow \mathcal{A}^I$ -- is subtler and depends delicately on the structure $\mathcal{A}$ and the ultrafilter $\mathcal{U}$. While for principal ultrafilters the embedding is trivial, for non-principal ultrafilters, the answer varies according to the algebraic properties of $\mathcal{A}$. To address this, we introduce some notions that capture the connectivity of elements of $\mathcal{A}$ under the operations and relations of the signature. We let the connected set of $a$ be composed by all the elements from which $a$ is obtained by means of an operation, or for which some operation is defined with $a$, or for which some relation holds with $a$. Let then $\llbracket a \rrbracket$ be the closure under connected sets of $a$. For a function $\mathrm{F} \in \sigma$, and $a \in A$, we define $S^{a \downarrow}_\mathrm{F}$ as the set of all tuples from which $a$ may be obtained by means of the function $\mathrm{F}^\mathcal{A}$; analogously, $S^{a \uparrow}_\mathrm{F}$ as the set of all the tuples containing $a$ for which $\mathrm{F}^\mathcal{A}$ is defined; and similarly, $S^a_\mathrm{R}$ as the set of all the tuples containing $a$ for which $\mathrm{R}^\mathcal{A}$ obtains. We then let $\tau_a$ be the sum of $|S^{a \downarrow}_\mathrm{F}|$, $|S^{a \uparrow}_\mathrm{F}|$ and $|S^a_\mathrm{R}|$ for all $\mathrm{F}, \mathrm{R} \in \sigma$, and $\tau_{\llbracket a \rrbracket}$ be the sum of all the $\tau_b$ for $b \in \llbracket a \rrbracket$. Finally, let $\tau_\mathcal{A}$ be the supremum of al the $\tau_{\llbracket a \rrbracket}$ for all $a \in A$.

In that regard, our main results are the following.

\setcounter{section}{3}
\setcounter{theorem}{5}

\begin{theorem}
If $\mathcal{U}$ has the $\tau_{\mathcal{A}^I / \mathcal{U}}^+$-intersection property, then $e : \mathcal{A}^I / \mathcal{U} \hookrightarrow \mathcal{A}^I$.
\end{theorem}

In Proposition \ref{proposition tau value}, we offer some upper bounds of the value of $\tau_{\mathcal{A}^I}$ and $\tau_{\mathcal{A}^I / \mathcal{U}}$.

For a $\sigma$-structure $\mathcal{A}$ and $a \in A$, let the operational set of $a$ be the set of all the elements which may be obtained from $a$ by means of some operation. Define $\mathrm{Cl}_\mathcal{A}(a)$ as the closure under operational sets of $a$. Let also a have a unique representation if whenever $\mathrm{F}(b_0, ..., b_n)$ and $\mathrm{G}(c_0, ..., c_m)$, then $\mathrm{F} = \mathrm{G}$ and $b_i = c_i$. Then:

\setcounter{section}{3}
\setcounter{theorem}{10}

\begin{theorem}
Let $\sigma$ have no relations and $\mathcal{A}$ be a $\sigma$-structure whose operations are total and such that each of its elements has a unique representation. If there is $\{a_\alpha\}_{\alpha < \kappa} \subseteq \mathcal{A}^I / \mathcal{U}$ such that $\{\mathrm{Cl}_\mathcal{B}(a_\alpha) \rangle_{\alpha < \kappa}$ partitions $\mathcal{A}^I / \mathcal{U}$ and $\{\mathrm{c}^{\mathcal{A}^I / \mathcal{U}} \mid \mathrm{c}\ \text{is a constant of}\ \sigma\} \subseteq \{a_\alpha\}_{\alpha < \kappa}$, then there is a choice function $e$ on the equivalence classes of $A^I / \mathcal{U}$ such that $e : \mathcal{A}^I / \mathcal{U} \hookrightarrow \mathcal{A}^I$.
\end{theorem}

\setcounter{section}{1}
\setcounter{theorem}{0}

Before we begin, we start with a few remarks on conventions we shall use. For $n \in \omega$, a set $A$ and $a \in A$, we write $\langle x_1, ..., a, ..., x_{n-1} \rangle$ for an $n$-tuple of $A^n$ in which $a$ occurs in some position. We shall also follow the convention of denoting structures by a calligraphic font, and their corresponding domains, by an italicized font -- for example, $\mathcal{A}$ and $A$, respectively. Let $\mathcal{A}$ be a structure and $I$ be a set. We denote by $\mathcal{A}^I$ the direct power of $A$ with index set $I$, and by $A^I$ its domain, that is, the set of all functions from $I$ into $A$. For a given element $a \in A$, we denote by $\overline{a}$ the constant function $a$ of $A^I$, when the index set $I$ is clear by context. If $\mathcal{U}$ is an ultrafilter over $I$, we denote by $\mathcal{A}^I / \mathcal{U}$ the resulting ultrapower. By the \emph{natural embedding} of a set or structure into its direct power or ultrapower, we mean the mapping which takes each element of the domain of the set and maps it to either the constant function which assigns to every element of the index set the element of the domain in question, in the case of direct powers, or its corresponding equivalence class modulo the ultrafilter, in the case of ultrapowers. We let $\kappa$, $\tau$, $\xi$, $\zeta$, and so on, stand for cardinals. For a cardinal $\kappa$, we denote by $\kappa^+$ its successor.

We take the basic first-order language to be composed of individual variables, negation, conjunction, equality, and the universal quantifier (where the other connectives and quantifiers may be defined as usual). For a signature $\sigma$ and an operation $\mathrm{F}$ or relation $\mathrm{R}$, we write $\mathrm{ar}(\mathrm{F})$ and $\mathrm{ar}(\mathrm{R})$ for their arity. As usual, we use a calligraphic letter $\mathcal{A}$ to denote a structure, and the respective italicized letter $A$ to denote its corresponding domain. Given a constant $\mathrm{c}$, function symbol $\mathrm{F}$ or relation symbol $\mathrm{R}$ of $\sigma$, and a $\sigma$-structure $\mathcal{A}$, we denote by $\mathrm{c}^\mathcal{A}$, $\mathrm{F}^\mathcal{A}$ and $\mathrm{R}^\mathcal{A}$ their interpretations in $\mathcal{A}$, respectively. Given an ultrafilter $\mathcal{U}$ quotienting a structure $\mathcal{A}^I$ and $a \in A^I$, we denote by $[a]_\mathcal{U}$ the equivalence class of $a$. Furthermore, throughout the paper, we shall not make the assumption the operations of a given structure are total, so that some may be, in principle, partial.

\section{Embedding between ultrapowers}

Suppose we have $\mathcal{A} \hookrightarrow \mathcal{B}$. When is it the case that $\mathcal{A}^I / \mathcal{U}_I \hookrightarrow \mathcal{B}^J / \mathcal{U}_J$? It is well known that $\mathcal{A}^I / \mathcal{U}$ is isomorphic to $\mathcal{A}$ when either $\mathcal{A}$ is finite or $\mathcal{U}$ is principal. Therefore, it is trivial that if $\mathcal{A} \hookrightarrow \mathcal{B}$ and either $\mathcal{A}$ and $\mathcal{B}$ are finite or $\mathcal{U}_I$ and $\mathcal{U}_J$ are principal, then $\mathcal{A}^I / \mathcal{U}_I \hookrightarrow \mathcal{B}^J / \mathcal{U}_J$. It is less clear, however, under what conditions embeddability may be lifted to the ultrapowers given both structures are infinite and the ultrapower of the embedding structure is non-isomorphic to its generating structure. The matter is naturally related to the Rudin-Keisler ordering of ultrafilters.

\begin{definition}[Rudin-Keisler ordering \cite{rudin1971}]
Let $\mathcal{U}_I$ and $\mathcal{U}_J$ be ultrafilters over $I$ and $J$. The Rudin-Keisler ordering of ultrafilters $\leq_\mathrm{RK}$ is defined by $\mathcal{U}_I \leq_\mathrm{RK} \mathcal{U}_J$ when there is $h : J \to I$ such that $X \in \mathcal{U}_I$ iff $h^{-1}[X] \in \mathcal{U}_J$. Two ultrafilters are Rudin-Keisler equivalent, written $\mathcal{U}_I \approx_\mathrm{RK} \mathcal{U}_J$, when $\mathcal{U}_I \leq_\mathrm{RK} \mathcal{U}_J$ and $\mathcal{U}_J \leq_\mathrm{RK} \mathcal{U}_I$.\footnote{In fact, that is only a preorder. The Rudin-Keisler (partial) ordering is defined on types of ultrafilters, which are defined as the equivalence classes of ultrafilters under the defined equivalence. For brevity, we use the concept as defined.}
\end{definition}

The relevance of the Rudin-Keisler ordering of ultrafilters is particularly relevant given the following result by Blass:

\begin{proposition}[\cite{blass}]\label{proposition rudin keisler}
Let $\mathcal{U}_I$ and $\mathcal{U}_J$ be ultrafilters over $I$ and $J$. Then the following are equivalent:

\begin{itemize}
\item[\emph{(1)}] $\mathcal{U}_I \leq_{\mathrm{RK}} \mathcal{U}_J$;
\item[\emph{(2)}] for any $\sigma$ and $\sigma$-structure $\mathcal{A}$, $\mathcal{A}^I / \mathcal{U}_I \hookrightarrow \mathcal{A}^J / \mathcal{U}_J$;
\item[\emph{(3)}] for any $\sigma$ and $\sigma$-structure $\mathcal{A}$, $\mathcal{A}^I / \mathcal{U}_I \preceq \mathcal{A}^J / \mathcal{U}_J$.
\end{itemize}
\end{proposition}

Since it is trivial to check that:

\begin{lemma}\label{lemma embeddings}
Let $\mathcal{A} \hookrightarrow \mathcal{B}$ and $\mathcal{U}$ be an ultrafilter over $I$. Then, $\mathcal{A}^I / \mathcal{U} \hookrightarrow \mathcal{B}^I / \mathcal{U}$.\footnote{To see so, just notice the identity map on $\mathcal{A}^I / \mathcal{U}$ is an embedding into $\mathcal{B}^I / \mathcal{U}$.}
\end{lemma}

We therefore obtain:

\begin{theorem}\label{theorem ultrapowers embedding}
Let $\mathcal{A} \hookrightarrow \mathcal{B}$ and $\mathcal{U}_I$ and $\mathcal{U}_J$ be ultrafilters over $I$ and $J$. If $\mathcal{U}_I \leq_{\mathrm{RK}} \mathcal{U}_J$, then $\mathcal{A}^I / \mathcal{U}_I \hookrightarrow \mathcal{B}^J / \mathcal{U}_J$.
\end{theorem}

\begin{proof}
By Proposition \ref{proposition rudin keisler}, $\mathcal{A}^I / \mathcal{U}_I \hookrightarrow \mathcal{A}^J / \mathcal{U}_J$, and by Lemma \ref{lemma embeddings}, $\mathcal{A}^J / \mathcal{U}_J \hookrightarrow \mathcal{B}^J / \mathcal{U}_J$, so that $\mathcal{A}^I / \mathcal{U}_I \hookrightarrow \mathcal{B}^J / \mathcal{U}_J$.
\end{proof}

Therefore, we have a sufficient criterion for the embeddability of $\mathcal{A}$ into $\mathcal{B}$ to be lifted to an embeddability of $\mathcal{A}^I / \mathcal{U}_I$ into $\mathcal{B}^J / \mathcal{U}_J$.

Despite the criterion provided by the Rudin-Keisler ordering being sufficient for embeddability to be raised to ultrapowers, it is unclear what are necessary and sufficient conditions for the raising of embeddability. We may now find stronger criteria of embeddability between ultrapowers which, in fact, do not require the embeddability of their generating structures.

For the upcoming results, we shall assume $|A| > 1$, for otherwise embeddability between ultrapowers is trivial.

\begin{definition}
Let $I$ and $J$ be sets, $H: \mathcal{P}(I) \to \mathcal{P}(J)$, and $X, Y \in \mathcal{P}(I)$. We say $H$ is:

\begin{itemize}[align=parleft, labelsep=8mm,]
\item[(a)] \emph{multiplicative} if $H(X) \cap H(Y) = H(X \cap Y)$;
\item[(b)] \emph{additive} if $H(X) \cup H(Y) = H(X \cup Y)$;
\item[(c)] \emph{subtractive} if $H(I \setminus X) = H(I) \setminus H(X)$.
\item[(d)] \emph{covering} if $H(I) = J$.
\end{itemize}
\end{definition}

\begin{proposition}\label{proposition h varnothing}
If $H: \mathcal{P}(I) \to \mathcal{P}(J)$ is subtractive and multiplicative, then $H(\varnothing) = \varnothing$.
\end{proposition}

\begin{proof}
Since $H$ is subtractive and multiplicative, $H(\varnothing) = H(\varnothing \cap I) = H(\varnothing) \cap H(I) = \linebreak H(\varnothing) \cap \big{(} H(I) \setminus H(\varnothing) \big{)} = \varnothing$.
\end{proof}

\begin{proposition}\label{proposition function ultrafilters partition}
Let $H: \mathcal{P}(I) \to \mathcal{P}(J)$ be additive, multiplicative, subtractive and covering. If $\{X_i\}_{i < \kappa}$ is a partition of $I$, then $\{H(X_i)\}_{i < \kappa}$ is a partition of $J$. 
\end{proposition}

\begin{proof}
Since $X_i \cap X_j = \varnothing$ for $i \neq j$ and $H$ is multiplicative and subtractive, by Proposition \ref{proposition h varnothing}, $H(X_i) \cap H(X_j) = H(X_i \cap X_j) = H(\varnothing) = \varnothing$. Since $H$ is additive and covering, $j \in J = H(I) = H(\bigcup_{i < \kappa} X_i) = \bigcup_{i < \kappa} H(X_i)$.
\end{proof}

\begin{lemma}\label{lemma ultrapowers embedding}
Let $H: \mathcal{P}(I) \to \mathcal{P}(J)$ be additive, multiplicative, subtractive and covering, and $u: A \to B$ be an injective mapping. For each $f \in A^I$, let $w(f) \in B^J$ be the function such that for each $a \in A$,

\begin{center}
 $\{j \in J \mid w(f)(j) = u(a)\} = H(\{j \in I \mid f(j) = a\})$.
\end{center}

\noindent
Then, for $f, g \in A^I$,

\begin{itemize}
\item[] $H(\{j \in I \mid f(j) = g(j)\}) = \{j \in J \mid w(f)(j) = w(g)(j)\}$, and
\item[] $H(\{j \in I \mid f(j) \neq g(j)\}) = \{j \in J \mid w(f)(j) \neq w(g)(j)\}$.
\end{itemize}
\end{lemma}

\begin{proof}
Notice $w(f)$ is well defined because $u$ is injective, and that by Proposition \ref{proposition function ultrafilters partition}, $w(f)$ is indeed unique. Furthermore, we have

\begin{center}
$\{j \in I \mid f(j) = g(j)\} = \bigcup_{a \in A} (\{j \in I \mid f(j) = a\} \cap \{j \in I \mid g(j) = a\})$,
\end{center}

\noindent
so that $H(\{j \in I \mid f(j) = g(j)\}) =$

\medskip

\begin{tabular}{ll}
$=$ & $H \big{(} \bigcup_{a \in A} (\{j \in I \mid f(j) = a\} \cap \{j \in I \mid g(j) = a\}) \big{)}$\\
$=$ & $\bigcup_{a \in A} H(\{j \in I \mid f(j) = a\} \cap \{j \in I \mid g(j) = a\})$\\
$=$ & $\bigcup_{a \in A} \big{(} H(\{j \in I \mid f(j) = a\}) \cap H(\{j \in I \mid g(j) = a\}) \big{)}$\\
$=$ & $\bigcup_{a \in A} (\{j \in J \mid w(f)(j) = u(a)\}) \cap \{j \in J \mid w(g)(j) = u(a)\})$\\
$=$ & $\{j \in J \mid w(f)(j) = w(g)(j)\}$,
\end{tabular}

\medskip

\noindent
since $\mathrm{img} \big{(} w(f) \big{)} = \mathrm{img} \big{(} w(g) \big{)} = u[A]$. Similarly, we have

\medskip

\begin{tabular}{lll}
$H(\{j \in I \mid f(j) \neq g(j)\})$ & $=$ & $H(I \setminus \{j \in I \mid f(j) = g(j)\})$\\
 & $=$ & $J \setminus H(\{j \in I \mid f(j) = g(j)\})$\\
 & $=$ & $J \setminus \{j \in J \mid w(f)(j) = w(g)(j)\}$\\
 & $=$ & $\{j \in J \mid w(f)(j) \neq w(g)(j)\}$.
\end{tabular}

\medskip

\end{proof}

\begin{lemma}\label{lemma ultrapowers embedding 2}
Let $h: J \to I$ be a mapping and $H: \mathcal{P}(I) \to \mathcal{P}(J); X \mapsto h^{-1}[X]$. Then $H$ is additive, multiplicative, subtractive and covering.
\end{lemma}

\begin{proof} 
Clearly, $H(I) = h^{-1}[I] = J$, so $H$ is covering. To see it is multiplicative, notice

\medskip

\begin{tabular}{lll}
$H(X \cap Y)$ & $=$ & $h^{-1}[X \cap Y]$\\
 & $=$ & $\{j \in J \mid h(j) \in X \cap Y\}$\\
 & $=$ & $\{j \in J \mid h(j) \in X\} \cap \{j \in J \mid h(j) \in Y\}$\\
 & $=$ & $h^{-1}[X] \cap h^{-1}[Y]$\\
 & $=$ & $H(X) \cap H(Y)$\\
\end{tabular}

\medskip

\noindent
By a similar argument, we may see it is additive and subtractive.
\end{proof}

\begin{theorem}\label{theorem ultrapowers embedding 2}
Let $\mathcal{A}$ and $\mathcal{B}$ be $\sigma$-structures, $u: \mathcal{A} \hookrightarrow \mathcal{B}$, $h: J \to I$ be a mapping, and $\mathcal{U}_I$ and $\mathcal{U}_J$ be ultrafilters over $I$ and $J$. Let also $H: X \mapsto h^{-1}[X]$ and define $w: A^I \to B^J$ as in \emph{Lemma \ref{lemma ultrapowers embedding}}. Define $e: \mathcal{A}^I / \mathcal{U}_I \to \mathcal{B}^J / \mathcal{U}_J ; [f]_{\mathcal{U}_I} \mapsto [w(f)]_{\mathcal{U}_J}$. Then $e$ is an embedding iff  $X \in \mathcal{U}_I \Leftrightarrow h^{-1}(X) \in \mathcal{U}_J$.
\end{theorem}

\begin{proof}
For simplicity, we write $\mathcal{A}'$ for $\mathcal{A}^I / \mathcal{U}_I$ and $\mathcal{B}'$ for $\mathcal{B}^J / \mathcal{U}_J$. Notice by Lemma \ref{lemma ultrapowers embedding 2}, $H$ is multiplicative, additive, subtractive and covering, so we may indeed define $w$ as in Lemma \ref{lemma ultrapowers embedding}. Now, for a constant $\mathrm{c}$, $w(\overline{\mathrm{c}^\mathcal{A}}) = \overline{\mathrm{c}^\mathcal{B}}$, so $e(\mathrm{c}^{\mathcal{A}'}) = e([\overline{\mathrm{c}^\mathcal{A}}]_{\mathcal{U}_I}) = [\overline{\mathrm{c}^\mathcal{B}}]_{\mathcal{U}_J} = \mathrm{c}^{\mathcal{B}'}$. Let now $\mathrm{F}^{\mathcal{A}'}([f_0]_{\mathcal{U}_I}, ..., [f_n]_{\mathcal{U}_I}) = [g]_{\mathcal{U}_I}$. One might check that $w \big{(} \mathrm{F}^{\mathcal{A}^I}(f_0, ..., f_n) \big{)} = \mathrm{F}^{\mathcal{B}^J} \big{(} w(f_0), ..., w(f_n) \big{)}$, so we get

\medskip

\begin{tabular}{lll}
$e \big{(} \mathrm{F}^{\mathcal{A}'}([f_0]_{\mathcal{U}_I}, ..., [f_n]_{\mathcal{U}_I}) \big{)}$ & $=$ & $e \big{(} [\mathrm{F}^{\mathcal{A}^I}(f_0, ..., f_n)]_{\mathcal{U}_I} \big{)}$\\
 & $=$ & $\big{[} w \big{(} \mathrm{F}^{\mathcal{A}^I}(f_0, ..., f_n) \big{)} \big{]}_{\mathcal{U}_J}$\\
 & $=$ & $\big{[} \mathrm{F}{\mathcal{B}^J} \big{(} w(f_0), ..., w(f_n) \big{)} \big{]}_{\mathcal{U}_J}$\\
 & $=$ & $\mathrm{F}^{\mathcal{B}'} \big{(} [w(f_0)]_{\mathcal{U}_J}, ..., [w(f_n)]_{\mathcal{U}_J} \big{)}$\\
 & $=$ & $\mathrm{F}^{\mathcal{B}'} \big{(} e([f_0]_{\mathcal{U}_I}), ..., e([f_n]_{\mathcal{U}_I}) \big{)}$. 
\end{tabular}

\medskip

\noindent
Therefore, $e$ preserves the operations and constants for any choice of mappings $u$ (as long as it is injective) and $h$. We now proceed to showing the remaining of each direction of the equivalence.

($\Rightarrow$) Suppose $e$ is an embedding. For each $X \subseteq I$, choose a unique pair $f_X, g_X \in A^I$ such that $X = \{j \in I \mid f_X(j) = g_X(j)\}$. Since $A^I$ contains every possible permutation, there are always such $f_X$ and $g_X$. By Lemma \ref{lemma ultrapowers embedding}, $H: X \mapsto \{j \in J \mid w(f_X)(j) = w(g_X)(j)\}$. That means \linebreak $X = \{j \in I \mid f_X(j) = g_X(j)\} \in \mathcal{U}_I$ iff $[f_X]_{\mathcal{U}_I} = [g_X]_{\mathcal{U}_I}$ iff $[w(f_X)]_{\mathcal{U}_J} = [w(g_X)]_{\mathcal{U}_J}$ iff \linebreak $\{j \in J \mid w(f_X)(j) = w(g_X)(j)\} = H(X) = h^{-1}[X] \in \mathcal{U}_J$.

($\Leftarrow$) Let $X \in \mathcal{U}_I$ iff $H(X) \in \mathcal{U}_J$. Let now $\mathrm{R}^{\mathcal{A}'}[f_0]_{\mathcal{U}_I} ... [f_n]_{\mathcal{U}_I}$. Then $\{j \in I \mid \mathrm{R}^\mathcal{A} f_0(j) ... f_n(j)\} \in \mathcal{U}_I$. For convenience, call that set $Z$, and let $k \in Z$ and $g_0, ..., g_n$ be such that $g_i(j) = f_i(j)$ for any $j \in Z$, and $g_i(j) = f_i(k)$ otherwise. Then, we get $\{j \in I \mid \mathrm{R}^\mathcal{A} g_0(j) ... g_n(j)\} = I$, and $Z \subseteq \{j \in I \mid g_i(j) = f_i(j)\} \in \mathcal{U}_I$. That means we have $\{j \in J \mid \mathrm{R}^\mathcal{B} w(g_0)(j) ... w(g_n)(j)\} = J \in \mathcal{U}_J$, and by Lemma \ref{lemma ultrapowers embedding} and our supposition, $H(\{j \in I \mid g_i(j) = f_i(j)\}) = \{j \in J \mid w(g_i)(j) = w(f_i)(j)\} \in \mathcal{U}_J$, so that

\begin{center}
$\{j \in J \mid \mathrm{R}^\mathcal{B} w(g_0)(j) ... w(g_n)(j)\} \cap (\bigcap_{i \leq n} \{j \in J \mid w(g_i)(j) = w(f_i)(j)\}) \subseteq \{j \in J \mid \mathrm{R}^\mathcal{B} w(f_0)(j) ... w(f_n)(j)\} \in \mathcal{U}_J$.
\end{center}

\noindent
Thus, $\mathrm{R}^{\mathcal{B}'}[w(f_0)]_{\mathcal{U}_J}...[w(f_n)]_{\mathcal{U}_J}$, that is, that $\mathrm{R}^{\mathcal{B}'} e \big{(} [f_0]_{\mathcal{U}_I} \big{)} ... e \big{(} [f_n]_{\mathcal{U}_I} \big{)}$. The other direction (that is, $\mathrm{R}^{\mathcal{B}'} e \big{(} [f_0]_{\mathcal{U}_I} \big{)} ... e \big{(} [f_n]_{\mathcal{U}_I} \big{)} \Rightarrow \mathrm{R}^{\mathcal{A}'}[f_0]_{\mathcal{U}_I} ... [f_n]_{\mathcal{U}_I}$) follows analogously. Suppose now $[f]_{\mathcal{U}_I} \neq [g]_{\mathcal{U}_I}$, so that $\{j \in I \mid f(j) \neq g(j)\} \in \mathcal{U}_I$. That means with Lemma \ref{lemma ultrapowers embedding} and our supposition, we have $\{j \in J \mid w(f)(j) \neq w(g)(j)\} = H(\{j \in I \mid f(j) \neq g(j)\}) \in \mathcal{U}_J$. Therefore, $e([f]_{\mathcal{U}_I}) = [w(f)]_{\mathcal{U}_J} \neq [w(g)]_{\mathcal{U}_J} = e([g]_{\mathcal{U}_I})$.
\end{proof}

Notice, if it were the case that any embedding between ultrapowers is of the form defined in Theorem \ref{theorem ultrapowers embedding 2}, then given $\mathcal{A} \hookrightarrow \mathcal{B}$, $X \in \mathcal{U}_I \Leftrightarrow h^{-1}(X) \in \mathcal{U}_J$ becomes a necessary and sufficient criterion for the embeddability of $\mathcal{A}^I / \mathcal{U}_I$ into $\mathcal{B}^J / \mathcal{U}_J$. Notice that fails to be the case in general.\footnote{In the empty language, if $A$ and $B$ are countably infinite, $I = J = \omega$, and $\mathcal{U}_I$ and $\mathcal{U}_J$ are non-principal ultrafilters over $\omega$, there are $2^{2^\omega}$ embeddings from $\mathcal{A}^\omega / \mathcal{U}_I$ into $\mathcal{B}^\omega / \mathcal{U}_J$, but only $2^\omega$ of them can be induced by maps $u: A \to B$ and $h: \omega \to \omega$.} As Blass has shows \cite{blass}, not every elementary embedding between $\mathcal{A}^I / \mathcal{U}_I$ and $\mathcal{A}^J / \mathcal{U}_J$ are induced by a morphism between $\mathcal{U}_I$ and $\mathcal{U}_J$: elementary embeddings induced by such morphisms are natural, while isomorphisms between saturated structures, for example, tend to be unnatural. Thus, we might raise the following problem:

\medskip

\begin{itemize}
\item[] \textbf{Open Problem 1:} Given that $\mathcal{A} \hookrightarrow \mathcal{B}$, is every \emph{definable} embedding $e: \mathcal{A}^I / \mathcal{U}_I \to \mathcal{B}^J / \mathcal{U}_J$ of the form $[f]_{\mathcal{U}_I} \mapsto [w(f)]_{\mathcal{U}_J}$ for some $u: \mathcal{A} \hookrightarrow \mathcal{B}$ and mapping $h: J \to I$? If not, what is a necessary and sufficient criterion for embeddability between $\mathcal{A}^I / \mathcal{U}_I$ and $\mathcal{B}^J / \mathcal{U}_J$?\footnote{For a similar question, with a positive answer, on the form of elementary embeddings between $\mathcal{A}^I / \mathcal{U}_I$ and $\mathcal{A}^J / \mathcal{U}_J$, one may see \cite{blass}, p. 85, Proposition 5.}
\end{itemize} 

\medskip

\section{Embedding ultrapowers into direct powers}

When does an ultrapower embed into its corresponding direct power, that is, when is it the case that $\mathcal{A}^I / \mathcal{U} \hookrightarrow \mathcal{A}^I$? The answer to the question depends on the ultrafilter of choice. For example, if $\mathcal{A}$ is an ordered structure and $\mathcal{U}$ is non-principal, then $\mathcal{A}^I / \mathcal{U}$ has elements which dominate the image of the natural embedding. However, no matter the choice of representatives, there is no element in the direct power $\mathcal{A}^I$ which assumes a value greater than the values of every representative of the equivalence classes of the image of the natural embedding. On the other hand, if $\mathcal{U}$ is principal, the embedding is trivial.

To try to answer that, we start by considering the functions of a structure. Notice the issue with the direct power's not being a field is there being functions with $0$ in their image, but which are distinct from $\overline{0}$. So, if one wants to retain multiplicative inverses, one should exclude any function containing elements without a multiplicative inverse in its image, that is, any function whose support is not its whole domain. However, for the desired isomorphism, there should still be a representative of each element not possessing multiplicative inverses. To generalize that for arbitrary structures, we start with the following definitions. 

For the next results, we let $\sigma$ be a signature, $\mathcal{A}$ be a $\sigma$-structure, $I$ be an index set, and $\mathcal{U}$ be an ultrafilter over it.

\begin{definition}[Support]\label{def support}
Define the set $\mathrm{supp}(\mathcal{A})$, the \emph{support of $\mathcal{A}$}, such that

\medskip

\begin{center}
$a \in \mathrm{supp}(\mathcal{A})$ iff $\mathrm{F}(a, ..., x_{\mathrm{ar}(\mathrm{F})}), ..., \mathrm{F}(x_1, ..., a, ... x_{\mathrm{ar}(\mathrm{F})}), ..., \mathrm{F}(x_1, ..., a)$ are defined for all $x_1, ..., x_{\mathrm{ar}(\mathrm{F})} \in A$ and $\mathrm{F} \in \sigma$.
\end{center}

\medskip

\noindent
Let $\overline{\mathrm{supp}}(\mathcal{A}) = A \setminus \mathrm{supp}(\mathcal{A})$, the \emph{anti-support of $\mathcal{A}$}. Define

\medskip

\begin{center}
$\mathrm{rep}(\mathcal{A}^I) = \{\overline{a} \in A^I \mid a \in \overline{\mathrm{supp}}(\mathcal{A})\} \cup \{f \in A^I \mid \forall j \in I (f(j) \in \mathrm{supp}(\mathcal{A}))\}$.
\end{center}
\end{definition}

The above definitions generalize the notions of support, and its counterpart, of a structure, and of a function. $\mathrm{rep}(\mathcal{A})$ is divided following the reasoning that the elements of $\mathcal{A}^I / \mathcal{U}$ that are in its support should find representatives among the functions in the support of $\mathcal{A}^I$ (which are the functions whose supports are their whole domains), and elements not in the support of $\mathcal{A}^I / \mathcal{U}$ should find representatives among the constant functions whose images are singletons of elements that are not in the support of their domain -- intuitively, reflecting also the fact that each element not in the support of $\mathcal{A}^I / \mathcal{U}$ should be equivalent, in a sense, to some element in the anti-support of $\mathcal{A}$. In essence, the definition of $\mathrm{rep}(\mathcal{A})$ captures the intuition that the ultrapower should not multiply the elements in the anti-support of its generating structure when they are definable. The interest in the above concepts is when $\mathcal{A}$ possesses partial functions, such as the rationals with division.

\begin{proposition}[\cite{changkeisler}, p. 228]\label{proposition kreisel}
$\mathcal{U}$ is $\kappa$-complete iff for every partition of $I$ into fewer than $\kappa$ parts, one of the parts belongs to $\mathcal{U}$.
\end{proposition}

\begin{theorem}\label{theorem representatives}
Every element of $\mathcal{A}^I / \mathcal{U}$ has a representative in $\mathrm{rep}(\mathcal{A}^I)$ iff $\mathcal{U}$ is $\kappa^+$-complete, where $\kappa = |\overline{\mathrm{supp}}(\mathcal{A})|$.
\end{theorem}

\begin{proof}
This proof is similar to the one for Proposition 4.2.4 of \cite{changkeisler}. ($\Leftarrow$) Let $f \in A^I$. If $\mathrm{supp}(f) = I$, the conclusion is straightforward, so suppose otherwise. Now, the sets $\mathrm{supp}(f)$ and $\{j \in I \mid f(j) = a\}_{a \in \overline{\mathrm{supp}}(\mathcal{A})}$ partition $I$. Since $\mathcal{U}$ is $\kappa^+$-complete for $\kappa = |\overline{\mathrm{supp}}(\mathcal{A})|$, by Proposition \ref{proposition kreisel} exactly one of those sets is in $\mathcal{U}$. But then either $[f] = [\overline{b}]$ for some $b \in \overline{\mathrm{supp}}(\mathcal{A})$, or $\mathrm{supp}(f) \in \mathcal{U}$, so consider $g \in A^I$ such that for some $c \in \mathrm{supp}(\mathcal{A})$,

\begin{center}
$g(j) = \begin{cases} f(j),\ \text{if}\ j \in \mathrm{supp}(\mathcal{A})\\
c,\ \text{otherwise}
\end{cases}$ 
\end{center}

\noindent
so that $[g] = [f]$. Either way, we have the conclusion, since $\overline{b}, g \in \mathrm{rep}(\mathcal{A}^I)$.

($\Rightarrow$) Suppose every $[f] \in \mathcal{A}^I / \mathcal{U}$ has a representative in $\mathrm{rep}(\mathcal{A}^I)$. Let $\bigcup_{\alpha < \xi} X_\alpha$ be a partition of $A$, for $\xi \leq \kappa = |\overline{\mathrm{supp}}(\mathcal{A})|$. Let $\langle a_\alpha \rangle_{\alpha < \kappa}$ be an enumeration of $\overline{\mathrm{supp}}(\mathcal{A})$ and $f \in A^I$ be such that $f(j) = a_\zeta$ iff $j \in X_\zeta$. By $f$'s construction, $[f] \neq [b]$ for any $b \in \mathrm{supp}(\mathcal{A})$, so by assumption we must have $[f] = [\overline{c}]$ for some $c \in \overline{\mathrm{supp}}(\mathcal{A})$. But by construction, $c = a_\zeta$ for some $\zeta < \xi$, which means $\{j \in I \mid f(j) = c\} = \{j \in I \mid f(j) = a_\zeta\} = X_\zeta \in \mathcal{U}$. By the arbitrariness of the partition and Proposition \ref{proposition kreisel}, $\mathcal{U}$ is $\kappa^+$-complete.
\end{proof}

Every ultrafilter is finitely complete, which justifies the set $\mathrm{rep}(\mathcal{A})$ always containing some representative of every equivalence class of $\mathcal{A}^I / \mathcal{U}$ when $|\overline{\mathrm{supp}} \big{(} \mathcal{A}) \big{)}| < \aleph_0$. However, as it is well known, the existence of $\kappa$-complete ultrafilters, for $\kappa > \aleph_0$, is independent of ZFC, which means the existence of such a fragment in $\mathrm{rep}(\mathcal{A})$, within ZFC, is only guaranteed for structures with $|\overline{\mathrm{supp}}(\mathcal{A})| < \aleph_0$:

\begin{corollary}
Let $|\overline{\mathrm{supp}}(\mathcal{A})| \geq \aleph_0$. Then, there is a non-principal $\mathcal{U}$ over $I$ such that $\mathrm{rep}(\mathcal{A}^I)$ contains a representative of every equivalence class of $\mathcal{A}^I / \mathcal{U}$ iff there is a measurable cardinal $\kappa > |\overline{\mathrm{supp}}(\mathcal{A})|$.
\end{corollary}

Nevertheless, the result also implies, within ZFC, that if $\mathcal{A}$ is a field (and, in the presence of additional operations $\{ \mathrm{F}_i \}_{i \leq \kappa}$, as long as $|A^{\mathrm{ar}(\mathrm{F_i})} \setminus \mathrm{dom}(\mathrm{F}_i)| < \aleph_0$), then $\mathrm{rep}(\mathcal{A})$ does contain a representative of each element of $\mathcal{A}^I / \mathcal{U}$.

\vspace{5mm}

For a $\sigma$-structure $\mathcal{A}$ and $a \in A$, let $\llbracket a \rrbracket$ be the \emph{connected closure set of $a$} -- that is, the set of all the elements in tuples from which $a$ may be obtained by means of an operation, or containing $a$ for which some function is defined, or containing $a$ for which some relation holds; the elements in tuples from which some element $b$ related in that way to $a$ may be obtained by means of an operation, or containing $b$ for which some function is defined, or containing $b$ for which some relation holds, and so on. That is, it may be defined, given an operation $\mathrm{cc} : A \to \mathcal{P}(A)$ such that $\mathrm{cc}(a) = \{b \mid \exists c_1, ..., c_{\mathrm{F}-1} (\mathrm{F}(c_1, ..., b, ..., c_{\mathrm{F}-1}) = a\ \text{or}\ \mathrm{F}(c_1, ..., a, ..., c_{\mathrm{F}-1}) = b)\ \text{or}\ \exists c_1, ..., c_{\mathrm{R}-2}(\mathrm{R} c_1 ... a ... b ... c_{\mathrm{R}-2})\}$, as the closure of $\{a\}$ under $\mathrm{cc}$.\footnote{That is, $\llbracket a \rrbracket = \bigcup_{n < \omega} \mathrm{cc}^n(\{a\})$, where $\mathrm{cc}^0(\{a\}) = \{a\}$ and $\mathrm{cc}^{n+1}(\{a\}) = \mathrm{cc}(\mathrm{cc}^n(\{a\}))$.} Let also $\tau_a$ be the cardinality of the set of all the tuples from which $a$ may be obtained by means of an operation, or containing $a$ for which some function is defined, or containing $a$ for which some relation holds. That is, for

\vspace{2mm}

\begin{tabular}{lll}
$S^{a \downarrow}_\mathrm{F}$ & $=$ & $\{\langle b_1, ..., b_{\mathrm{ar}(\mathrm{F})} \rangle \in A^{\mathrm{ar}(\mathrm{F})} \mid \mathrm{F}^\mathcal{A}(b_1, ..., b_{\mathrm{ar}(\mathrm{F})}) = a\}$,\\
$S^{a \uparrow}_\mathrm{F}$ & $=$ & $\{\langle b_1, ..., b_{\mathrm{ar}(\mathrm{F})} \rangle \in A^{\mathrm{ar}(\mathrm{F})} \mid \mathrm{F}^\mathcal{A}(b_1, ..., a, ..., b_{\mathrm{ar}(\mathrm{F})-1}) = b_{\mathrm{ar}(\mathrm{F})}\}$, and\\
$S^a_\mathrm{R}$ & $=$ & $\{\langle b_1, ..., b_{\mathrm{ar}(\mathrm{R})-1} \rangle \in A^{\mathrm{ar}(\mathrm{R})-1} \mid \mathrm{R}^\mathcal{A} b_1 ... a ... b_{\mathrm{ar}(\mathrm{R})-1}\}$,
\end{tabular}

\vspace{2mm}

\noindent
we let $\tau_a = \big{|} \bigsqcup_{\mathrm{F} \in \sigma} (S^{a \downarrow}_\mathrm{F} \sqcup S^{a \uparrow}_\mathrm{F}) \sqcup  \bigsqcup_{\mathrm{R} \in \sigma} S^a_\mathrm{R} \big{|} = \sum_{\mathrm{F} \in \sigma} (|S^{a \downarrow}_\mathrm{F}| + |S^{a \uparrow}_\mathrm{F}|) + \sum_{\mathrm{R} \in \sigma} |S^a_\mathrm{R}|$. Let then $\tau_{\llbracket a \rrbracket} = \sum_{b \in \llbracket a \rrbracket} \tau_b$, and define $\tau_\mathcal{A} = \mathrm{sup}\{\tau_{\llbracket a \rrbracket}\}_{a \in A}$.

\begin{proposition}\label{proposition connected closure}
For $a, b \in \mathcal{A}$, if $\llbracket a \rrbracket \neq \llbracket b \rrbracket$, then $\llbracket a \rrbracket \cap \llbracket b \rrbracket = \varnothing$.
\end{proposition}

\begin{proof}
Straightforward.
\end{proof}

Recall that, for a cardinal $\kappa$, a family of sets $X$ has the $\kappa$-intersection property if any subset $Y \subseteq X$ with $|Y| < \kappa$ is such that $\bigcap Y \neq \varnothing$.

\begin{theorem}\label{theorem representatives embedding}
If $\mathcal{U}$ has the $\tau_{\mathcal{A}^I / \mathcal{U}}^+$-intersection property, then there is $e : \mathcal{A}^I / \mathcal{U} \hookrightarrow \mathcal{A}^I$.
\end{theorem}

\begin{proof}
Call $\mathcal{B} = \mathcal{A}^I / \mathcal{U}$ and let $\theta: B \to A^I$ be a choice function on the equivalence classes. Let $\langle a_\alpha \rangle_{\alpha < |B|}$ be an enumeration of the elements of $\mathcal{B}$. We inductively define the following relation. Let $e_0 = \varnothing$. Then, for each $a \in B$, for $e_{< \beta} = \bigcup_{\alpha < \beta} e_\alpha$ and each $c \in \llbracket a \rrbracket \cup \{a\}$, let

\begin{center}
$c^\dagger = \begin{cases}
e_{< \beta}(c),\ \text{if}\ a \in \mathrm{dom}(e_{< \beta})\\
\theta(c)\ \text{otherwise}
\end{cases}$
\end{center}

\noindent
In that way, for each $\langle b_1, ..., b_{\mathrm{ar}(\mathrm{F})} \rangle \in S^{c \downarrow}_\mathrm{F}$,

\begin{center}
$\big{\{} j \in I \mid \mathrm{F}^\mathcal{A} \big{(} b_1^\dagger(j), ..., b^\dagger_{\mathrm{ar}(\mathrm{F})}(j) \big{)} = c^\dagger(j) \big{\}} \in \mathcal{U}$;
\end{center}

\noindent
similarly, for each $\langle b_1, ..., b_{\mathrm{ar}(\mathrm{F})} \rangle \in S^{c \uparrow}_\mathrm{F}$,

\begin{center}
$\big{\{} j \in I \mid \mathrm{F}^\mathcal{A} \big{(} b_1^\dagger(j), ..., c^\dagger(j), ..., b_{\mathrm{ar}(\mathrm{F})-1}^\dagger(j) \big{)} = b_{\mathrm{ar}(\mathrm{F})}^\dagger(j) \big{\}} \in \mathcal{U}$,
\end{center}

\noindent
and for each $\langle b_1, ..., b_{\mathrm{ar}(\mathrm{R}-1)} \rangle \in S^c_\mathrm{R}$,

\begin{center}
$\big{\{} j \in I \mid \mathrm{R}^\mathcal{A} b_1^\dagger(j) ... c^\dagger(j) ... b_{\mathrm{ar}(\mathrm{F})-1}^\dagger(j) \big{\}} \in \mathcal{U}$.
\end{center}

\noindent
Since and $\mathcal{U}$ has the $\tau^+_\mathcal{B}$-intersection property and $\tau_{\llbracket a \rrbracket} \leq \tau_\mathcal{B}$,

\begin{center}
$\bigcap_{c \in \llbracket a \rrbracket} \Big{(} \bigcap_{\mathrm{F} \in \sigma} \big{(} \bigcap_{\langle b_1, ..., b_{\mathrm{ar}(\mathrm{F})} \rangle \in S_\mathrm{F}^{c \downarrow} } \big{\{} j \in I \mid \mathrm{F}^\mathcal{A} \big{(} b_1^\dagger(j), ..., b^\dagger_{\mathrm{ar}(\mathrm{F})}(j) \big{)} = c^\dagger(j) \big{\}}\ \cap\ \linebreak \bigcap_{\langle b_1, ..., b_{\mathrm{ar}(\mathrm{F})} \rangle \in S_\mathrm{F}^{c \uparrow}} \big{\{} j \in I \mid \mathrm{F}^\mathcal{A} \big{(} b_1^\dagger(j), ..., c^\dagger(j), ..., b_{\mathrm{ar}(\mathrm{F})-1}^\dagger(j) \big{)} = b_{\mathrm{ar}(\mathrm{F})}^\dagger(j) \big{\}} \big{)} \cap \linebreak \bigcap_{\langle b_1, ..., b_{\mathrm{ar}(\mathrm{F})-1} \rangle \in S_\mathrm{R}^c} \big{\{} j \in I \mid \mathrm{R}^\mathcal{A} b_1^\dagger(j) ... c^\dagger(j) ... b_{\mathrm{ar}(\mathrm{R})-1}^\dagger(j) \big{\}} \Big{)} \neq \varnothing$.
\end{center}

\noindent
Call the above set $X^{\llbracket a \rrbracket}$ and let $k \in X^{\rrbracket a \llbracket}$. Define then $f_c$ such that for each $j \in I$,

\begin{center}
$f_c(j) = \begin{cases}
\theta(c)(j),\ \text{if}\ j \in X^{\rrbracket a \llbracket}\\
\theta(c)(k),\ \text{otherwise}
\end{cases}$
\end{center}

\noindent
Then, we let $e_\beta = \bigcup_{\alpha < \beta} e_\alpha \cup \{ \langle c, f_c \rangle \mid c \in (\llbracket a \rrbracket \cup \{a\})\}$. Notice, by Proposition \ref{proposition connected closure}, $e_\beta$ is a function.

From the above construction, it is clear that $e = \bigcup_{\alpha < |B|} e_\alpha$ is a function, that it preserves the constants, functions and relations of $\sigma$, and that for any $a \in B$, $[e(a)] = a$.
\end{proof}

\begin{corollary}\label{corollary representatives embedding}
If $\mathcal{U}$ is $\tau_{\mathcal{A}^I / \mathcal{U}}^+$-complete, then there is a choice function $e$ on the equivalence classes of $\mathcal{A}^I / \mathcal{U}$ such that $e : \mathcal{A}^I / \mathcal{U} \hookrightarrow \mathcal{A}^I$.
\end{corollary}

\begin{proof}
By Theorem \ref{theorem representatives embedding}, noting that if $\mathcal{U}$ is $\tau_{\mathcal{A}^I / \mathcal{U}}^+$-complete, then for each $a \in \mathcal{A}^I / \mathcal{U}$, $[f_a] = a$.
\end{proof}

We note that the converse is not in general true. Let $\mathcal{A}$ be an $\omega$-dimensional vector space over $\mathbb{Q}$ with signature $\sigma = \{+\} \cup \{\mathrm{F}_q \mid q \in \mathbb{Q}\}$, where $+$ is vector addition and $\mathrm{F}_q$ is a unary function for scalar multiplication by $q$. For any $a \in \mathcal{A}$, $|S^{a \downarrow}_+| = |A|$, and $|S^{a \downarrow}_{\mathrm{F}_q}| = |A|$ if $a = \overline{0}$, and $|S^{a \downarrow}_{\mathrm{F}_q}| = 1$ if $a \neq 0$ and $q \neq 0$. In that case, we may see $\tau_\mathcal{A} = |A|$, and $\tau_{\mathcal{A}^I / \mathcal{U}} = |\mathcal{A}^I / \mathcal{U}|$. Let now $I = \omega$ and $\mathcal{U}$ be non-principal. Then the dimensions of $\mathcal{A}^\omega$ and $\mathcal{A}^\omega / \mathcal{U}$ are the same, so $\mathcal{A}^\omega \cong \mathcal{A}^\omega / \mathcal{U}$ as vector spaces. However, $|\mathcal{A}^\omega / \mathcal{U}| \geq \aleph_0$, so $\tau^+_{\mathcal{A}^\omega / \mathcal{U}} \geq \aleph_1$. But since $\mathcal{U}$ is a non-principal ultrafilter over $\omega$, it cannot have the $\aleph_1$-intersection property.

\begin{proposition}\label{proposition tau value}
Let $\mathcal{A}$ be a $\sigma$-structure whose relations are empty. Suppose $\tau_\mathcal{A}$ is not a limit cardinal, let $\mathrm{m} \in A$ be the element such that $\tau_\mathcal{A} = \tau_{\llbracket \mathrm{m} \rrbracket}$, and suppose for any $n, n' \in \llbracket m \rrbracket$ and $\mathrm{F} \in \sigma$, $|S^{n \uparrow}_\mathrm{F}| = |S^{n' \uparrow}_\mathrm{F}|$ and $|S^{n \downarrow}_\mathrm{F}| = |S^{n' \downarrow}_\mathrm{F}|$. For $Z \in \mathcal{U}$ such that $|Z| = \mathrm{min}\{|X|\}_{X \in \mathcal{U}}$:

\begin{itemize}
\item[\emph{(a)}] $\tau_{\mathcal{A}^I}\ =\ \sum_{n \in \llbracket m \rrbracket} \sum_{\mathrm{F} \in \sigma} \big{(} |S^{\mathrm{n} \downarrow}_\mathrm{F}|^{|I|} + |S^{\mathrm{n} \uparrow}_\mathrm{F}|^{|I|} \big{)}\ \leq\ |\llbracket m \rrbracket|  \times (\tau_\mathcal{A})^{|I|}$;
\item[\emph{(b)}] $\tau_{\mathcal{A}^I / \mathcal{U}}\ \leq\ \mathrm{max} \big{\{} \sum_{b \in \llbracket a \rrbracket} \sum_{\mathrm{F} \in \sigma} \big{(} \prod_{j \in Z} |S^{b(j) \downarrow}_\mathrm{F}| + \prod_{j \in Z} |S^{b(j) \uparrow}_\mathrm{F}| \big{)} \big{\}}_{a \in A^I}\ \leq\ |\llbracket m \rrbracket| \times (\tau_\mathcal{A})^{|Z|}$.
\end{itemize}
\end{proposition}

\begin{proof}
(a): Let $a \in \mathcal{A}^I$. Then

\begin{center}
\begin{tabular}{lll}
$S^{a \downarrow}_\mathrm{F}$ & $=$ & $\{\langle b_1, ..., b_{\mathrm{ar}(\mathrm{F})} \rangle \in (\mathcal{A}^I)^{\mathrm{ar}(\mathrm{F})} \mid \mathrm{F}^{\mathcal{A}^I} (b_1, ..., b_{\mathrm{ar}(\mathrm{F})}) = a \}$\\
 & $=$ & \small $\big{\{} \langle b_1, ..., b_{\mathrm{ar}(\mathrm{F})} \rangle \in (\mathcal{A}^I)^{\mathrm{ar}(\mathrm{F})} \mid \forall j \in I \big{(} \mathrm{F}^\mathcal{A} \big{(} b_1(j), ..., b_{\mathrm{ar}(\mathrm{F})}(j) \big{)} = a(j) \big{)} \big{\}}$ \normalsize \\
 & $=$ & \small $\big{\{} \langle b_1, ..., b_{\mathrm{ar}(\mathrm{F})} \rangle \in (\mathcal{A}^I)^{\mathrm{ar}(\mathrm{F})} \mid \forall j \in I \big{(} \langle b_1(j), ..., b_{\mathrm{ar}(\mathrm{F})}(j) \rangle \in S^{a(j) \downarrow}_\mathrm{F} \big{)} \big{\}}$ \normalsize \\
\end{tabular}
\end{center}

\noindent
and likewise for $S^{a \uparrow}_\mathrm{F}$. Thus,

\begin{center}
$|S^{a \downarrow}_\mathrm{F}| = \Pi_{j \in I} |S^{a(j) \downarrow}_\mathrm{F}|$
\end{center}

\noindent
(and likewise for $|S^{a \uparrow}_\mathrm{F}|$). We may notice that, by assumption, we have

\begin{center}
\begin{tabular}{lll}
$\tau_{\mathcal{A}^I}$ & $=$ & $\tau_{\llbracket \overline{\mathrm{m}} \rrbracket}$\\
 & $=$ & $|\bigsqcup_{n \in \llbracket \overline{m} \rrbracket} \bigsqcup_{\mathrm{F} \in \sigma} \big{(} S^{n \downarrow}_\mathrm{F} \sqcup S^{n \uparrow}_\mathrm{F} \big{)}|$\\
 & $=$ & $\sum_{n \in \llbracket m \rrbracket} \sum_{\mathrm{F} \in \sigma} \big{(} \Pi_{j \in I} |S^{n \downarrow}_\mathrm{F}| + \Pi_{j \in I} |S^{n \uparrow}_\mathrm{F}| \big{)}$\\
 & $=$ & $\sum_{n \in \llbracket m \rrbracket} \sum_{\mathrm{F} \in \sigma} \big{(} |S^{n \downarrow}_\mathrm{F}|^{|I|} + |S^{n \uparrow}_\mathrm{F}|^{|I|} \big{)}\ \leq\ |\llbracket m \rrbracket| \times \big{(} \sum_{\mathrm{F} \in \sigma} |S^{n \downarrow}_\mathrm{F}| + |S^{n \uparrow}_\mathrm{F}| \big{)}^{|I|}\ =\ |\llbracket m \rrbracket| \times (\tau_\mathcal{A})^{|I|}$.
\end{tabular}
\end{center}

\vspace{2mm}

(b): Let $a \in \mathcal{A}^I / \mathcal{U}$ and $\cdot^\dagger : \mathcal{A}^I / \mathcal{U} \to A^I$ be a choice function on the equivalence classes. Then, as we have seen,

\begin{center}
$|S^{a^\dagger \downarrow}_\mathrm{F}|\ =\ \Pi_{j \in I} \big{|} S^{a^\dagger(j) \downarrow}_\mathrm{F} \big{|}$
\end{center}

\noindent
(and likewise for $|S^{a^\dagger \uparrow}_\mathrm{F}|$). Let $X \subseteq I$. Notice

\begin{center}
\small $\big{|} \big{\{} \langle b_1, ..., b_{\mathrm{ar}(\mathrm{F})} \rangle \in (A^I)^{\mathrm{ar}(\mathrm{F})} \mid \{j \in I \mid \langle b_1(j), ..., b_{\mathrm{ar}(\mathrm{F})}(j) \rangle \in S^{a^\dagger(j) \downarrow}_\mathrm{F} \} = X \big{\}} \big{|}$ \normalsize $=$\\ $(A^{\mathrm{ar}(\mathrm{F}) \times |I \setminus X|}) \times \prod_{j \in X} |S^{a^\dagger(j) \downarrow}_\mathrm{F}|$.
\end{center}

\noindent
Thus, we may see, if $X \in \mathcal{U}$,

\begin{center}
 $\big{|} \big{\{} \langle [b_1]_\mathcal{U}, ..., [b_{\mathrm{ar}(\mathrm{F})}]_\mathcal{U} \rangle \in (A^I / \mathcal{U})^{\mathrm{ar}(\mathrm{F})} \mid \{j \in I \mid \langle b_1(j), ..., b_{\mathrm{ar}(\mathrm{F})}(j) \rangle \in S^{a^\dagger(j) \downarrow}_\mathrm{F} \} = X \big{\}} \big{|}$ $\leq$\\ $\prod_{j \in X} |S^{a^\dagger(j) \downarrow}_\mathrm{F}|$,
\end{center}

\noindent
so that

\begin{center}
$|S^{a \downarrow}_\mathrm{F}| = |\{\langle [b_1], ... [b_n] \rangle \mid \{j \in I \mid \langle b_1^\dagger(j), ..., b_n^\dagger \rangle \in S^{a(j) \downarrow}_\mathrm{F}\} \in \mathcal{U}\}|\ \leq\ \prod_{j \in Z} |S^{a^\dagger(j) \downarrow}_\mathrm{F}|$,
\end{center}

\noindent
and likewise for $S^{a \uparrow}_\mathrm{F}$. Thus,

\begin{center}
$\tau_{\mathcal{A}^I / \mathcal{U}} \leq \mathrm{max} \big{\{} \sum_{b \in \llbracket a \rrbracket} \sum_{\mathrm{F} \in \sigma} \big{(} \prod_{j \in Z} |S^{b^\dagger(j) \downarrow}_\mathrm{F}| + \prod_{j \in Z} |S^{b^\dagger(j) \uparrow}_\mathrm{F}| \big{)} \big{\}}_{a \in A^I / \mathcal{U}}$,
\end{center}

\noindent
and by the arbitrariness of $\cdot^\dagger$, we may conclude

\begin{center}
$\tau_{\mathcal{A}^I / \mathcal{U}} \leq \mathrm{max} \big{\{} \sum_{b \in \llbracket a \rrbracket} \sum_{\mathrm{F} \in \sigma} \big{(} \prod_{j \in Z} |S^{b(j) \downarrow}_\mathrm{F}| + \prod_{j \in Z} |S^{b(j) \uparrow}_\mathrm{F}| \big{)} \big{\}}_{a \in A^I}$.
\end{center}

\noindent
Therefore, similarly to the proof of (a),

\begin{center}
\begin{tabular}{lll}
$\tau_{\mathcal{A}^I / \mathcal{U}}$ & $\leq$ & $\mathrm{max} \big{\{} \sum_{b \in \llbracket a \rrbracket} \sum_{\mathrm{F} \in \sigma} \big{(} \prod_{j \in Z} |S^{b(j) \downarrow}_\mathrm{F}| + \prod_{j \in Z} |S^{b(j) \uparrow}_\mathrm{F}| \big{)} \big{\}}_{a \in A^I}$\\
 & $\leq$ & $|\llbracket m \rrbracket| \times \big{(} \sum_{\mathrm{F} \in \sigma} |S^{\mathrm{m} \downarrow}_\mathrm{F}| + |S^{\mathrm{m} \uparrow}_\mathrm{F}| \big{)}^{|Z|}$\\
 & $=$ & $|\llbracket m \rrbracket| \times (\tau_\mathcal{A})^{|Z|}$.
\end{tabular}
\end{center}
\end{proof}

An interesting consequence of Theorem \ref{theorem representatives embedding} is the following result:

\begin{corollary}\label{corollary representatives embedding 3}
Let $\mathcal{A}$ be a $\sigma$-structure with finite many function symbols whose relations are empty. Let also $\tau_\mathcal{A}$ not be a limit cardinal. If for any $\mathrm{F} \in \sigma$ and $a \in A$, $|S^{a \downarrow}_\mathrm{F}|, |S^{a \uparrow}_\mathrm{F}| \leq 1$, for any $a \in A$ either $S^{a \downarrow}_\mathrm{F} = \varnothing$ or $S^{a \uparrow}_\mathrm{F} = \varnothing$, and for any $a \in A$, $\llbracket a \rrbracket < \aleph_0$ then there is a choice function $e$ on the equivalence classes of $A^I / \mathcal{U}$ such that $e : \mathcal{A}^I / \mathcal{U} \hookrightarrow \mathcal{A}^I$.
\end{corollary}

\begin{proof}
By Proposition \ref{proposition tau value},

\begin{center}
$\tau_{\mathcal{A}^I / \mathcal{U}}\ \leq\ \mathrm{max} \big{\{} \sum_{b \in \llbracket a \rrbracket} \sum_{\mathrm{F} \in \sigma} \big{(} \prod_{j \in Z} |S^{b(j) \downarrow}_\mathrm{F}| + \prod_{j \in Z} |S^{b(j) \uparrow}_\mathrm{F}| \big{)} \big{\}}_{a \in A^I}\ \leq\ \sum_{b \in \llbracket a \rrbracket} \sum_{\mathrm{F} \in \sigma} 1^{|Z|}\ <\ \aleph_0$,
\end{center}

\noindent
so since every ultrafilter has the finite intersection property, by Theorem \ref{theorem representatives embedding} we have our result.
\end{proof}

For a $\sigma$-structure $\mathcal{A}$ and $a \in A$, let $\mathrm{Cl}_\mathcal{A}(a)$ be the \emph{operational closure set of $a$ in $\mathcal{A}$} -- that is, the set of all the elements which may be obtained from $a$ by means of some operation, the elements which may be obtained from those elements, and so on. For an element $a \in A$, let also a \emph{representation} of it be some application $\mathrm{F}(b_0, ..., b_n)$ for some function $\mathrm{F} \in \sigma$ and $b_0, ..., b_n \in A$. Say $a \in A$ has a \emph{unique representation} if whenever $\mathrm{F}(b_0, ..., b_n)$ and $\mathrm{G}(c_0, ..., c_m)$, then $\mathrm{F} = \mathrm{G}$ and $b_i = c_i$.

\begin{lemma}\label{lemma unique representations}
Let $\sigma$ be a finite signature and $\mathcal{A}$ be a $\sigma$-structure. If each element in $\mathcal{A}$ has a unique representation, then each element in $\mathcal{A}^I / \mathcal{U}$ has a unique representation.
\end{lemma}

\begin{proof}
Having a unique representation is first-order definable by the schema, for every distinct $\mathrm{F} \in \sigma$, $\forall x_1 ... x_n \forall y_1 ... y_n \big{(} \mathrm{F}(x_1, ..., x_n) = \mathrm{F}(y_1, ..., y_n) \to \bigwedge_{i = 1}^n x_i = y_i \big{)}$, and for every distinct $\mathrm{F}, \mathrm{G} \in \sigma$ of arities respectively $n, m$, $\forall x_1 ... x_n \forall y_1 ... y_m \big{(} \mathrm{F}(x_1, ..., x_n) \neq \mathrm{G}(y_1, ..., y_m) \big{)}$. Thus, the conclusion follows by \L o\'{s}'s theorem.
\end{proof}

\begin{theorem}\label{theorem representatives embedding 3}
Let $\sigma$ have no relations and $\mathcal{A}$ be a $\sigma$-structure whose operations are total and such that each of its elements has a unique representation. If there is $\{a_\alpha\}_{\alpha < \kappa} \subseteq \mathcal{A}^I / \mathcal{U}$ such that $\{\mathrm{Cl}_\mathcal{B}(a_\alpha) \rangle_{\alpha < \kappa}$ partitions $\mathcal{A}^I / \mathcal{U}$ and $\{\mathrm{c}^{\mathcal{A}^I / \mathcal{U}} \mid \mathrm{c}\ \text{is a constant of}\ \sigma\} \subseteq \{a_\alpha\}_{\alpha < \kappa}$, then there is a choice function $e$ on the equivalence classes of $A^I / \mathcal{U}$ such that $e : \mathcal{A}^I / \mathcal{U} \hookrightarrow \mathcal{A}^I$.
\end{theorem}

\begin{proof}
Call $\mathcal{B} = \mathcal{A}^I / \mathcal{U}$ and let $\cdot^\dagger : B \to A^I$ be a choice function on the equivalence classes. Let $\{\mathrm{Cl}_\mathcal{B}(a_\alpha)\}_{\alpha < \kappa}$ partition $\mathcal{B}$. We inductively define the following function. For each $n < \omega$, we employ the following construction:

\begin{itemize}
\item[•] ($0$th step) let:

\vspace{2mm}

\begin{tabular}{lll}
$B_0$ & $=$ & $\{a_\alpha\}_{\alpha < \kappa}$,\\
$e_0$ & $=$ & $\{\langle a_\alpha, a^\dagger_\alpha \rangle \mid a_\alpha \in \{a_\alpha\}_{\alpha < \kappa}\ \text{is not a constant}\}\ \cup$\\ 
 & & $\{\langle a_\alpha, \overline{a_\alpha} \rangle \mid a_\alpha \in \{a_\alpha\}_{\alpha < \kappa}\ \text{is a constant}\}$.
\end{tabular}

\vspace{2mm}

\item[•] ($n+1$th step) for $e_{< \beta} = \bigcup_{\alpha < \beta} e_\alpha$:

\vspace{2mm}

\begin{tabular}{lll}
$B_{n+1}$ & $=$ & $\{\mathrm{F}^\mathcal{B}(b_0, ..., b_m) \mid b_0, ..., b_m \in B_n\}$,\\
$e_{n+1}$ & $=$ & $\{\langle \mathrm{F}^\mathcal{B}(b_0, ..., b_m), \mathrm{F}^{\mathcal{A}^I} \big{(} e_{< \beta}(b_0), ..., e_{< \beta}(b_m) \big{)} \rangle \mid b_0, ..., b_m \in B_n\}$.\\
\end{tabular}
\end{itemize}

\noindent
Notice, at most, $B_\omega = \mathcal{B}$. Now, let $e = \bigcup_{n \in \omega} e_\alpha$. We may notice $e$ is a function, since $\{\mathrm{Cl}_\mathcal{B}(a_\alpha)\}_{\alpha < \kappa}$ partitions $\mathcal{B}$ and, by Lemma \ref{lemma unique representations}, each element has a unique representation. It is clear $e$ is also injective, and preserves functions. By induction on the stages of the construction, we may also see $e$ is a choice function on the equivalence classes of $\mathcal{B}$.
\end{proof}

Notice the previous results offer conditions for embedding either directly based on the properties of $\mathcal{A}$ and $\mathcal{A}^I / \mathcal{U}$, or on the cardinal invariant $\tau_{\mathcal{A}^I / \mathcal{U}}$, which in turn depends on the structure of $\mathcal{A}^I / \mathcal{U}$. In fact, it is a well-known result of universal algebra that a structure embeds into a product iff it has a separating family of homomorphisms into the factors. In the specific case at hand, that becomes the following criterion.

\begin{proposition}\label{proposition necessary sufficient embedding}
There is an embedding $e: \mathcal{A}^I / \mathcal{U} \hookrightarrow \mathcal{A}^I$ if and only if there is a family of homomorphisms $\{h_i : \mathcal{A}^I / \mathcal{U} \to \mathcal{A}\}_{i \in I}$ such that for any \(a \neq b \in \mathcal{A}^I / \mathcal{U}\), there is an $i \in I$ such that $h_i(a) \neq h_i(b)$.
\end{proposition}

Let $\mathbb{N}$ and $\mathbb{Z}$ be, respectively, the naturals and integers in the signature $\{0, \mathrm{s}, \mathrm{p}\}$, where $\mathrm{s}$ and $\mathrm{p}$ are respectively the successor and predecessor functions.

\begin{corollary}
The following hold:

\begin{itemize}
\item[\emph{(1)}] $\mathbb{N}^I / \mathcal{U} \not\hookrightarrow \mathbb{N}^I$;
\item[\emph{(2)}] $\mathbb{Z}^I / \mathcal{U} \hookrightarrow \mathbb{Z}^I$. 
\end{itemize}
\end{corollary}

\begin{proof}
(1): Let $\mathbb{N}^* = \mathbb{N}^I / \mathcal{U}$. For simplicity, we use $\mathbb{N}$ to also refer to the image of the natural embedding. Let $n \in \mathbb{N}^* \setminus \mathbb{N}$. Suppose there is an homomorphism $h: \mathbb{N}^* \to \mathbb{N}$. We know $n \neq 0$, so that $\mathrm{p}(n)$ is defined, and that $\mathrm{p}(n) \in \mathbb{N}^* \setminus \mathbb{N}$. By induction, we may see for any $m \in \omega$, $\mathrm{p}^m(n) \in \mathbb{N}^* \setminus \mathbb{N}$ (where $\mathrm{p}^m(n)$ is the element obtained by applying $\mathrm{p}$ $m$ times to $n$). That means there is an infinite sequence $n, \mathrm{p}(n), ..., \mathrm{p}^m(n), ...$. Since $h$ is a homomorphism, for each $m \in \omega$, $h(\mathrm{p}^m(n)) = \mathrm{p}^m(h(n))$. Notice $\mathrm{p}^m(h(n)) = h(n) - m$. Since $h(n) \in \mathbb{N}$, that means $h(\mathrm{p}^{h(n)+1}(n)) = \mathrm{p}^{h(n)+1}(h(n)) = h(n) - (h(n) + 1)$, but $h(n) - (h(n) + 1)$ is undefined in $\mathbb{N}$. Thus, there is no such homomorphism, so the conclusion follows by Proposition \ref{proposition necessary sufficient embedding}.

(2): If $I$ is finite, the result is trivial, so suppose otherwise. Let $\mathbb{Z}^* = \mathbb{Z}^I / \mathcal{U}$. Let $Z$ be a galaxy of $\mathbb{Z}$.\footnote{We recall a galaxy is defined by a collection of elements whose difference is finite.} Let $a_Z \in Z$. Then each element in $Z$ may be represented by $\mathrm{s}^m(a_Z)$ for some $m \in \mathbb{Z}$, where a negative $m$ represents the application of the predecessor function. Let $h_0 : \mathbb{Z}^* \to \mathbb{Z}; \mathrm{s}^k(a_Z) \mapsto k$, where $Z$ is the galaxy of each input. Define an equivalence relation in $\mathbb{Z}^I$ such that $f \sim g$ iff $f -^{\mathbb{Z}^I} g$ is a constant function. Then each equivalence class has $|\mathbb{Z}| = \aleph_0$ elements, which means, since $I$ is infinite, there are $|\mathbb{Z}^I| = 2^{|I|}$ such equivalence classes. Let now $\mathbf{Z}$ be the collection of all the galaxies of $\mathbb{Z}^*$. We have $2^{|I|} \geq \mathbb{Z}^*$, which means $|\mathbf{Z}| \leq 2^{|I|}$. Therefore, there is an injective $v : \mathbf{Z} \to \mathbb{Z}^I$ such that $v(Z_0) = \overline{0}$, for $Z_0$ the galaxy of $0$, and such that for each $Z \neq Z'$, $v(Z) \not\sim v(Z')$. For each $i \in I$, define now $h_i : \mathbb{Z}^* \to \mathbb{Z}$ such that $h(\mathrm{s}^m(a_Z)) \mapsto v(Z)(i) + m$. We may easily notice each $h_i$ is an homomorphism. Let now $a, b \in \mathbb{Z}^*$. If $a, b \in Z$ for some $Z \in \mathbf{Z}$, then they have different representations, and thus clearly for any $i \in I$, $h_i(a) \neq h_i(b)$. If now $a \in Z$ and $b \in Z'$ for $Z \neq Z'$, let $a = \mathrm{s}^k(a_Z)$ and $b = \mathrm{s}^m(a_{Z'})$. Since $v(Z) \not\sim v(Z')$, $v(Z) -^{\mathbb{Z}^I} v(Z')$ is not constant, and thus there is $i \in I$ such that $v(Z)(i) - v(Z')(i) \neq k - m$. Since $h_i(a) = v(Z)(i) + k$ and $h_i(b) = v(Z') + m$, $h_i(a) - h_i(b) = v(Z)(i) + k - (v(Z')(i) + m) = (v(Z)(i) - v(Z')(i)) + (k - m) \neq 0$. Therefore, $h_i(a) \neq h_i(b)$. Since $I$ is infinite, let $\langle i_\kappa \rangle_{\kappa < |I|}$ be an enumeration of the $I$, and define the family $\{h'_i\}_{i \in I}$ such that $h'_{i_0} = h_0$, $h_{i_n}' = h_{i_{n-1}}$ for $0 < n < \omega$, and $h'_{i_\kappa} = h_{i_\kappa}$ otherwise. Then $\{h'_i\}_{i \in I}$ is a family of homomorphisms satisfying the conditions of Proposition \ref{proposition necessary sufficient embedding}.
\end{proof}

A question thus naturally rises: is embeddability of an ultrapower into its related direct power always dependent on each particular structure, or is there an ultrafilter property which offers a different necessary, or sufficient, criterion for embeddability? In fact, we pose the following problem:

\medskip

\begin{itemize}
\item[] \textbf{Open Problem 2:} Is there a sufficient condition on $\mathcal{U}$ for the embeddability $\mathcal{A}^I / \mathcal{U} \hookrightarrow \mathcal{A}^I$? In general, is there such a sufficient criterion for embeddability?
\end{itemize} 

\medskip

\noindent
We conjecture that no such criterion based solely on the properties of the ultrafilter exists.

We end the paper with a small remark on embedding between real closed fields. Let $\mathbb{Q}^*$ be some construction of the hyperrationals, and $\infty$ and $\varepsilon$ be its infinite and infinitesimal elements, respectively. It is well known that the quotient ring $(\mathbb{Q}^* \setminus \infty) / \varepsilon$ is isomorphic to $\mathbb{R}$. We show now how that more generally applies between ordered real closed fields and their subfields.

Recall, that, for a cardinal $\kappa$, an ultrafilter $\mathcal{U}$ over $I$ is $\kappa$-regular if there is $X \subseteq \mathcal{U}$ such that $|X| = \kappa$ and each $i \in I$ belongs to finitely many elements of $X$. Recall also that a structure $\mathcal{B}$ is called \emph{$\kappa$-universal} if every structure $\mathcal{A}$ of cadinality less than $\kappa$ which is elementarily equivalent to $\mathcal{B}$ elementarily embeds in $\mathcal{B}$.

\begin{proposition}[\cite{changkeisler}, p. 255]\label{proposition regular}
Suppose $|\sigma| < \kappa$ and $\mathcal{U}$ over $I$ is a $\kappa$-regular ultrafilter. Then for every structure $\mathcal{A}$, $\mathcal{A}^I / \mathcal{U}$ is $\kappa^+$-universal.
\end{proposition}

\begin{proposition}[\cite{changkeisler}, p. 250]\label{proposition regular 2}
Let $\mathcal{U}$ be a $\kappa$-regular ultrafilter over a set $I$. If $A$ is infinite, then $|\mathcal{A}^I / \mathcal{U}| = |A^I|$.
\end{proposition}

The theory of (ordered) real closed fields is complete. Thus, if $\mathcal{B}$ is a real closed field, $\mathcal{A} \subseteq \mathcal{B}$ is a real closed subfield, $\mathcal{U}$ is a $\kappa$-regular ultrafilter over $I$, and $|\mathcal{B}| \leq \kappa$, then $\mathcal{B} \preceq \mathcal{A}^I / \mathcal{U}$. That is:

\begin{remark}\label{theorem reals}
Let $\mathcal{B}$ be a (ordered) real-closed field, $\mathcal{A} \preceq \mathcal{B}$ be a (ordered) real-closed subfield with $|B| \leq |I|$. If $\mathcal{U}$ is a non-principal $|I|$-regular ultrafilter over $I$, then $\mathcal{B} \preceq \mathcal{A}^I / \mathcal{U}$.
\end{remark}

\begin{proof}
By Propositions \ref{proposition regular} and \ref{proposition regular 2}.
\end{proof}

Write $\overline{\mathbb{Q}}$ for the algebraic numbers.

\begin{corollary}
If $\mathcal{A} \hookrightarrow \mathbb{R}$ is an ordered real closed subfield, $|I| \geq 2^{\aleph_0}$ and $\mathcal{U}$ is $|I|$-regular, then $\mathbb{R} \preceq \mathcal{A}^A / \mathcal{U}$. Particularly, that means if $\mathcal{U}$ is $2^{\aleph_0}$-regular, $\mathbb{R} \preceq \overline{\mathbb{Q}}^{2^{\aleph_0}} /\mathcal{U}$ 
\end{corollary}

That is, the hyperalgebraic field $\overline{\mathbb{Q}}^{2^{\aleph_0}} / \mathcal{U}$, for a $2^{\aleph_0}$-regular $\mathcal{U}$, is a real closed non-Archimedean field extension of the reals.

\bibliography{hypersurreals}
\bibliographystyle{plain}

\end{document}